\documentclass[AMA,Times1COL]{WileyNJDv5} 

\articletype{Research Article}%

\received{00}
\revised{00}
\accepted{00}
\journal{Optimal Control Applications and Methods}
\volume{00}
\copyyear{2025}
\startpage{1}

\usepackage{tikz}

\newcommand{\one}{\mathbf{1}}

\newcommand {\R}{\mathbb{R}}

\newcommand{\I}{\mathcal{I}}

\newcommand{\iPhi}{\Phi_{\mathrm{isto}}}
\newcommand{\N}{\mathbb{N}}
\renewcommand{\v}{\bar{v}}
\newcommand{\M}{\bar{M}}
\newcommand{\D}{\Delta}
\newcommand{\K}{\mathcal{K}}
\begin{document}

\title{An Efficient Mixed-Integer Formulation and an Iterative Method for Optimal Control of Switched Systems Under Dwell Time Constraints}

\author[LeuvenMECO]{Ramin Abbasi-Esfeden}
\author[Freiburg]{Armin Nurkanovic}
\author[Freiburg]{Moritz Diehl}
\author[LeuvenESAT]{Panagiotis Patrinos}
\author[LeuvenMECO,LeuvenFland]{Jan Swevers}

\authormark{Abbasi-Esfeden \textsc{et al.}}
\titlemark{An Efficient Mixed-Integer Formulation and an Iterative Method for Optimal Control of Switched Systems Under Dwell Time Constraints}

\address[LeuvenMECO]{\orgdiv{Department of Mechanical Engineering}, \orgname{KU Leuven}, \orgaddress{\state{Leuven}, \country{Belgium}}}
\address[LeuvenFland]{\orgdiv{Department of Mechanical Engineering}, \orgname{Flanders Make@KU Leuven}, \orgaddress{\state{Leuven}, \country{Belgium}}}
\address[Freiburg]{\orgdiv{Department of Microsystems Engineering (IMTEK)}, \orgname{University of Freiburg}, \orgaddress{\state{Freiburg}, \country{Germany}}}
\address[LeuvenESAT]{\orgdiv{Department of Electrical Engineering (ESAT)}, \orgname{KU Leuven}, \orgaddress{\state{Leuven}, \country{Belgium}}}

\corres{Ramin Abbasi Esfeden. \email{ramin.abbasiesfeden@kuleuven.be}}

\presentaddress{Celestijnenlaan 300, 3001 Leuven, Belgium.}

\fundingInfo{DFG via Research Unit FOR 2401, project 424107692, 504452366 (SPP 2364), and 525018088, by BMWK via 03EI4057A and 03EN3054B. The project has also received funding from the European Union’s Horizon 2020 research and innovation program under the Marie Skłodowska-Curie agreement No.~953348.}

\abstract[Abstract]{This paper presents an efficient Mixed-Integer Nonlinear Programming (MINLP) formulation for systems with discrete control inputs under dwell time constraints. By viewing such systems as a switched system, the problem is decomposed into a Sequence Optimization (SO) and a Switching Time Optimization (STO)---the former providing the sequence of the switched system, and the latter calculating the optimal switching times. By limiting the feasible set of SO to subsequences of a master sequence, this formulation requires a small number of binary variables, independent of the number of time discretization nodes. This enables the proposed formulation to provide solutions efficiently, even for large numbers of time discretization nodes. To provide even faster solutions, an iterative algorithm is introduced to heuristically solve STO and SO. The proposed approaches are then showcased on four different switched systems and results demonstrate the efficiency of the MINLP formulation and the iterative algorithm.}

\keywords{Optimal Control, Switched Systems, Dwell Time Constraints, Mixed-Integer Nonlinear Programming.}

\jnlcitation{\cname{%
\author{Abbasi Esfeden R.},
\author{Nurkanovic A},
\author{Diehl M},
\author{Patrinos P}, and
\author{Swevers J}}.
\ctitle{An Efficient Mixed-Integer Formulation and an Iterative Method for Optimal Control of Switched Systems Under Dwell Time Constraints.} \cjournal{\it J Optim Contr Appl Met.} \cvol{00}.}

\maketitle

\renewcommand\thefootnote{}
\footnotetext{\textbf{Abbreviations:} STO, Switching Time Optimization; SO, Sequence Optimization; MINLP, Mixed-Integer Nonlinear Program; MDT, Minimum Dwell Time.}

\renewcommand\thefootnote{\fnsymbol{footnote}}
\setcounter{footnote}{1}

\section{Introduction}\label{sec:introduction}

Systems with discrete control inputs are common in real-world applications. However, due to the discrete nature of the control inputs, solving optimal control problem for these systems is challenging. Examples of such problems can be found in the benchmark library by Sager \cite{lee_benchmark_2012}. A common approach is to first discretize the continuous time problem and then optimize it using Mixed-Integer Nonlinear Programming (MINLP) techniques \cite{belotti_mixed-integer_2013}. However, for problems where a fine time discretization is necessary, this usually leads to a large number of integer variables and, consequently, high computation times. Another aspect of real-world applications that adds to the complexity of the problem are constraints known as Minimum Dwell Time (MDT) constraints.
They put a lower bound on the duration between two switches of the discrete control input and are often necessary due to physical limitations. Within the usual MINLP framework, formulation of such constraints is not trivial, often introducing additional integer variables, further complicating the original MINLP.

An alternative and efficient method is to approximate the discrete control inputs based on their continuous relaxation. One example is the Combinatorial Integral Approximation (CIA) \cite{sager_combinatorial_2011}. The properties of such approximation have been investigated in literature \cite{manns_relaxed_2021, bestehorn_combinatorial_2022}. Within CIA, simple MDT constraints can be managed during the approximation step of discrete inputs on a fixed time grid. However, this leads to lower accuracy of the integer approximation and a larger optimality gap \cite{zeile_mixed-integer_2021}, which for some systems can lead to poor outcomes \cite{abbasi-esfeden_iterative_2023}.

Systems with discrete control inputs can alternatively be analyzed as switched dynamical systems. A switched dynamical system consists of a set of subsystems and a schedule that indicates the active subsystem at each moment. The schedule includes a sequence of modes, each corresponding to an active subsystem, along with a set of switching times that govern transitions between modes. The optimal control of the original system now includes selecting the active subsystem within each mode and the optimal duration, or dwell time, of each mode. Due to this structure, the formulation of MDT constraints for a switched system is arguably more intuitive.

Switching Time Optimization (STO) is an approach for switched systems where for a fixed sequence of modes, the duration of each mode is optimized simultaneously with all the other optimization variables and constraints. Although fixing the sequence of modes is a major limitation of the STO problem, in the absence of MDT constraints, a reduction of some mode durations to zero can allow the system to visit different orders of modes \cite{stellato_second-order_2017}. However, the inclusion of MDT constraints will prevent such mode collapses. The quality of the solution will now depend critically on the initial fixed sequence.

Sequence Optimization (SO) is the problem of obtaining a sequence of modes for a switched system. This problem has been studied under different names, such as mode scheduling \cite{axelsson_gradient_2008, wardi_controlled_precision_2012, xuping_xu_optimal_2000}. These methods attempt to generate a sequence based on the sensitivity of the cost function to the insertion of new modes. However, the optimality conditions of mode scheduling become inapplicable when infinitesimal variations are prohibited by MDT constraints \cite{ali_optimal_2016}.

In this paper, we propose a MINLP formulation that combines SO and STO while incorporating MDT constraints. This formulation requires a small number of binary variables, independent of the number of time discretization nodes, which enables efficient solutions even for large number of time discretization nodes. Furthermore, an algorithm is suggested to solve SO and STO iteratively and efficiently.

Section~\ref{sec:formulation} starts by introducing the switched system we are considering and then introduces the STO and SO formulations. Section~\ref{sec:dtconstraints} explains how MDT constraints are handled. Section~\ref{sec:isto} introduces the Iterative Switching Time Optimization (ISTO) algorithm, and Section~\ref{sec:results} applies the proposed MINLP formulation and ISTO to well-known examples of switched systems.

{\bf Notation:} We will use $[a \,..\, b]$ for the integer interval between $a$ and $b$ and $[n]$ as a shorthand for $[1\,..\,n] = \{1,2,\ldots, n\}$.
\section{PROBLEM FORMULATION}
\label{sec:formulation}
\subsection{Discrete Input Systems}
\label{sec:switched}

The systems we are considering are described by the following Ordinary Differential Equation (ODE) for almost all $t \in [0,\, t_f]$:
\begin{align}
\dot{x}(t) &= f\big(x(t), u(t), \mathrm{v}(t)\big), 
\label{originalsys}
\end{align}
with  $x(t) \in \R^{n_x}$, $u(t) \in \R^{n_u}$, and $\mathrm{v}(t) \in \{ 1, 2, \ldots, n\}$ \footnote{Although \eqref{originalsys} introduces the discrete control input taking its values from $\{1,2,\ldots, n\}$, the formulation can easily be extended to more general discrete control inputs $v'(t)\in \{v'^{(1)}, v'^{(2)}, \ldots, v'^{(n)}\}$, with $v'^{(i)} \in \R^{n_{v'}}$ for $i \in [n]$. We only need to introduce a new input $v$ by $v(v'(t)) = i$ if $v'(t) = v'^{(i)}$ for $t \in [t_0, t_f]$, to transform the problem into \eqref{originalsys}. Therefore, the numerical results are formulated as the more general case $v(t)\in \{v^{(1)}, v^{(2)}, \ldots, v^{(n)}\}$.}. Function $\mathrm{v}$ is a piecewise constant scalar function with $n$ possible values. Assuming a finite number of switches, $\mathrm{v}$ as a piecewise constant function partitions the interval $[0, \, t_f]$ into a finite sequence of intervals, over which it has a constant value. Let $[t_k, \,t_{k+1})$ be the $k$th interval and define $v_k := \mathrm{v}(t), \forall t \in [t_k, \,t_{k+1})$. Each constant $v_k$ defines a subsystem with fixed discrete control inputs over $[t_k, \,t_{k+1})$ as follows:
\begin{align}
\dot{x}(t) = f(x(t), u(t), v_k). \label{switchedsys}
\end{align}
We will refer to the $[t_k, \,t_{k+1})$ interval as the $k$th \emph{mode} and define $w_k := t_{k+1} - t_{k}$ as the \emph{dwell time} of the mode $k$. The system \eqref{switchedsys} defines the subsystems of a switched system with $\mathrm{v}$ governing the switching from one subsystem to another. We can represent a switched system with $M$ modes as follows:
\begin{center}
\begin{tikzpicture}[>=latex]
\draw(0,0) -- (4,0);
\draw(5,0) -- (6,0);
\draw[thick, dotted](4,0) -- (5,0);
\foreach \x in {0, 2, 4, 5, 6}{
\draw[thick](\x,-0.15) -- (\x,0.15);
}
\node at (0, -0.7) {$t_1 = 0$};
\node at (2, -0.7) {$t_2$};
\node at (4, -0.7) {$t_3$};
\node at (5, -0.7) {$t_M$};

\node at (0, 0.4) {$x_1$};
\node at (2, 0.4) {$x_2$};
\node at (4, 0.4) {$x_3$};

\node[above] at (1,0.1) {$f(\cdot, \cdot, v_1)$};
\node[above] at (3,0.1) {$f(\cdot, \cdot, v_2)$};
\node[above] at (5.5,0.1) {$f(\cdot, \cdot, v_M)$};

\draw[<->] (0,-0.3) -- (2,-0.3);
\node[below] at (1,-0.3) {$w_1$};

\draw[<->] (2,-0.3) -- (4,-0.3);
\node[below] at (3,-0.3) {$w_2$};

\draw[<->] (5,-0.3) -- (6,-0.3);
\node[below] at (5.5,-0.3) {$w_M$};

\node[left] at (0,0) {$0$};
\node[right] at (6,0) {$t_f$};

\end{tikzpicture}
\end{center}
where $x_k := x(t_k)$. The control input $\mathrm{v}$ can be characterized using the sequence of its values $v = (v_1, v_2, \ldots, v_M) \in [n]^M$ and the dwell time of each mode $w = (w_1, w_2, \ldots, w_M)^\top\!\in\R_{\geq 0}^M$. 

For a given $v$, MDT constraints put bounds on the dwell times $w$. The restriction may be on the dwell time of a single mode or on the total dwell time of multiple consecutive modes. Let $I \subset [M]$ be a set of mode indices that are subject to a MDT constraint. In this paper, we will consider MDT constraints in the following form:
\begin{align}
\sum_{k \in I} w_k \geq \D, \label{eq:dtc}
\end{align}
with $\D \geq 0$. We will refer to $I$ as a \emph{segment of $v$} for the constraint \eqref{eq:dtc}. As a MDT constraint typically relates to more than one segment of $v$, we collect all the affected segments in a set $\I(v) = \{ I_1, I_2, \ldots, I_{n_\I} \}$, and the MDT constraint now reads as:
\begin{alignat}{2}
\sum_{k \in I} w_k \geq \D, && \qquad \forall I \in \I(v), \label{eq:dtc:g} 
\end{alignat}
We are now prepared to define the optimal control problem for the system described by \eqref{originalsys}. To make sure the problem is feasible, we make the following assumption on the dwell time constraints:
\begin{assumption}
\label{assum:dwell}
For any dwell time constraint in the form \eqref{eq:dtc:g}, we have $0<\Delta < t_f$.
\end{assumption}
Having access to an integrator for the dynamics in \eqref{switchedsys}, let $F(x_k, u_k, v_k, w_k)$ represent the function that transitions the system state $x_k \in \R^{n_x}$ from the beginning to the end of mode $k$ based on the control input trajectory $u_k = (u_{k,1}^\top, u_{k, 2}^\top, \ldots, u_{k, N_k}^\top)^\top \in \R^{N_kn_u}$, with $u_{k,\cdot} \in \R^{n_u}$, and the scalar dwell time of $kth$ mode $w_k \in \R_{\geq 0}$. Note that $N_k$ determines the number of parameters that parametrizes the trajectory $u_k$ over the $k$th mode. Similarly, let $L(x_k, u_k, v_k, w_k)$ be the value of the objective function for the mode $k$. Note that for any mode $k \in [M]$, $x_k \in \R^{n_x}, u_k \in \R^{N_kn_u}$, and $ v_k \in [n]$, we have $F(x_k, u_k, v_k, 0) = x_k $ and $L(x_k, u_k, v_k, 0) = 0$, i.e., modes with zero dwell time do not affect the system or the objective. 

Let the system be under a MDT constraint expressed by \eqref{eq:dtc:g}. The Switching Time Optimization (STO) problem for \eqref{originalsys} with its discrete control input characterized by the sequence $v$ is as follows:
\begin{subequations}
\begin{alignat}{2}
\Phi(v) = \min_{\substack{x, u, w}} &\sum_{k=1}^M L(x_k, u_k, v_k, w_k) + E(x_{M+1}) \\
\text{s.t.}\hspace{0.9cm} x_1 &= \bar{x}_1,\\
x_{k+1} &= F(x_k, u_k, v_k, w_k), && \hspace{-1cm}\forall k \in [M],   \\
\sum_{k=1}^M w_k &= t_f, \\
\sum_{k \in I} w_k &\geq \D, &&\hspace{-1cm} \forall I \in \I(v) \label{sto:dtc},
\end{alignat}
\label{sto}%
\end{subequations}
where $x = (x_1^\top, x_2^\top, \ldots, x_{M+1}^\top) \in \R^{(M+1)n_x}$, and $u = (u_1^\top, u_2^\top, \ldots, u_M^\top) \in \R^{MN_kn_u}$
with $x_k \in \R^{n_x}$, $v_k \in [n]$, $w_k \in \R_{\geq 0}$, and $u_k \in \R^{N_k n_u}$.  Problem \eqref{sto} is a Nonlinear Program (NLP), with no integer variables, as $v$ is fixed. The Sequence Optimization (SO) problem is written as:
\begin{align}
\min_{M \in \N, v \in [n]^M} \Phi(v). \label{so}
\end{align}
Problem \eqref{so} is a combinatorial optimization problem, which is not tractable in this general form. A change in $v$ can change all the elements of $\I(v)$ for the MDT constraint \eqref{sto:dtc}, leading to the appearing and vanishing of inequalities. This prevents us from applying a branch-and-bound algorithm, as it is not trivial to provide a bounding function for \eqref{so}. On the other hand, not including the dwell times in the computation of the lower bound will give a weak bounding function \cite{Clausen2003BranchAB}, as the lower bound will not agree with the objective function of a feasible solution of \eqref{so}. More importantly, the dimension of $v$, i.e., $M$, is not known beforehand.

To circumvent these difficulties, we will limit the feasible set of \eqref{so} to the set of subsequences of a \emph{master sequence} $\v \in [n]^{\M}$ for some $\M \in \N$. A binary vector $b \in \{0,1\}^{\M}$ can be associated with each subsequence of $\v$, such that $b_j = 1$ if and only if the mode $\v_j$ is in the corresponding subsequence. Leveraging the fact that modes with zero dwell times do not affect the system or the objective, it is possible to incorporate the master sequence into STO, and write SO only in terms of $b$ as follows:
\begin{align}
\min_{b \in \{0,1\}^{\M}} \phi(b), \label{so:final}
\end{align}
where $\phi(b)$ is similar to \eqref{sto} with its sequence $v$ fixed to the master sequence $\v = (\v_1, \v_2, \ldots, \v_{\M})$, and it is containing a constraint that sets $w_k = 0$ if $b_k = 0$ for all $k \in [\M]$, i.e., it reads as:
\begin{subequations}
\begin{alignat}{2}
\phi(b) = \min_{\substack{x, u, w}} &\sum_{k=1}^{\M} L(x_k, u_k, \v_k, w_k) + E(x_{\M+1}) \\
\text{s.t.}\hspace{0.9cm} x_1 &= \bar{x}_1,\\
x_{k+1} &= F(x_k, u_k, \v_k, w_k), &&\hspace{-1cm} \forall k \in [\M],   \\
\sum_{k=1}^{\M} w_k &= t_f, \\
0 \leq  w_k &\leq b_k T, &&\hspace{-1cm} \forall k \in [\M], \label{sto:dtc:b}\\
w \in &~S(b), \label{sto:dtc:placeholder}
\end{alignat}
\label{sto:binary}%
\end{subequations}
where $T \geq t_f$ is any upper bound for dwell times. Note that \eqref{sto:dtc:placeholder} is a placeholder for the MDT constraint, as \eqref{sto:dtc} is generally inconsistent with \eqref{sto:dtc:b}. In the following section, we explain how to modify \eqref{sto:dtc} and place it back in \eqref{sto:binary} to arrive at a well-posed formulation of $\phi(b)$.

\subsection{Minimum Dwell Time Constraints}
\label{sec:dtconstraints}

By fixing a master sequence $\v$, we were able to limit the feasible set of SO to the subsequences of $\v$. This was done through the binary vector $b$ and the constraint \eqref{sto:dtc:b}. If $b_k = 0$ for the $k$th mode, its dwell time shrinks to zero by \eqref{sto:dtc:b}, and the mode becomes ineffective on the system and the objective. However, to completely remove the impact of a mode, we need to look into its influence on the MDT constraints as well. We explain this via an example. Consider a system with the following master sequence for its discrete control input:
\begin{align}
\v = (2, 3, 2, 3, 2), \label{example:master}
\end{align}
and let there only be a minimum dwell time constraint on modes with the discrete control input of $2$, i.e., modes $1$, $3$, and $5$. In case $b_k = 1$ for all $k \in [5]$, \eqref{sto:dtc} becomes
\begin{align}
w_1 \geq \D, \; w_3 \geq \D, \; w_5 \geq \D. \label{example:dtc1}
\end{align}
However, in case $b_2 = 0$ and $b_k = 1$ for $k \in \{1,3,4,5\}$, corresponding to the subsequence $(2, 2, 3, 2)$ of $\v$, we want \eqref{sto:dtc} to become
\begin{align}
w_1 + w_3 \geq \D, \; w_5 \geq \D.
\end{align}
Setting $b_2 = 0$ needs \eqref{sto:dtc} to treat the second mode as nonexistent and join the adjacent modes. Furthermore, if $w_1 + w_3 \geq \D$, then $w_1 \geq \D$ and $w_3 \geq \D$ should not be imposed to avoid overconstraining the modes $1$ and $3$. 
We need a systematic way to handle such cases, and this can be done using the master sequence.

By a quick examination of $\v$ in \eqref{example:master}, we can foresee three cases of $b$ in which the modes of interest that were previously separated need to be joined:
\begin{subequations}
\label{example:cases}
\begin{align}
b^{(1)} &= (1, 1, 1, 0, 1), \quad b^{(2)} = (1, 0, 1, 1, 1), \quad b^{(3)} = (1, 0, 1, 0, 1),
\end{align}
\end{subequations}
corresponding to the subsequences $(2,3,2,2)$, $(2,2,3,2)$, and $(2,2,2)$, respectively. Note that these cases are based on whether $b_2 = 0$ and/or $b_4 = 0$ and are the only cases we need to investigate, as we only need to consider the largest possible joined segments of $\v$ for a MDT constraint. However, which segments of $\v$ these are will depend on $b$. 

To make this dependency explicit, first we need to extend the definition of $\I$ to include not only segments of $\v$ such as $\{1\}$, $\{3\}$, and $\{5\}$, but also its other segments due to the possible joining of modes. For \eqref{example:master} and its minimum dwell time constraint, this becomes
\begin{align}
\I = \{ \{1\}, \{3\}, \{5\}, \{3,5\}, \{1,3\}, \{1,3,5\} \}. \label{example:segments}
\end{align}
Elements $\{3,5\}$, $\{1,3\}$ and $\{1,3,5\}$ are indices of modes that are joined if $b$ takes the values of $b^{(1)}$, $b^{(2)}$, and $b^{(3)}$, respectively. Therefore, the elements of $\I$ that will actually become subject to the corresponding MDT constraint now depend on $b$. 

Let $z_I \in \{0, 1\}$ denote a binary variable for a segment $I \in \I$ of $\v$ for a MDT constraint that indicates whether the segment $I$ is subject to that MDT constraint or not. For instance $z_{\{1, 3\}}$ for $\{1, 3\} \in \I$ from \eqref{example:segments}, indicates whether two modes of $1$ and $3$ have become adjacent to each other due to the removal of the second mode, i.e., whether we need to put a constraint on $w_1 + w_3$ or not. The value of $z_I$ depends on $b$. Before going into details of this relation, we make the following definitions.
\begin{definition}
We refer to a mode $k$ as \emph{included} in a subsequence of $v$ if and only if $b_k = 1$. We refer to a segment $I$ of a sequence $v$ as \emph{active} under a MDT if and only if $z_I = 1$.
\end{definition}
Based on the definition of $z_I$, we see that $z_I$ is determined by $b$ through the following conditions: $z_I = 1$ if and only if
\begin{subequations}
\label{so:conditions}
\begin{alignat}{3}
&\phantom{and} \qquad b_i &&= 1, &&\qquad \forall i \in I, \label{so:condition1}\\
&\text{and} \qquad b_j &&= 0, &&\qquad \forall j \in [\underline{I}\, ..\, \overline{I}]\setminus I, \label{so:condition2}\\
&\text{and} \qquad z_{I'} &&= 0,  &&\qquad \forall I' \in \I ,\;I' \supset I. \label{so:condition3}
\end{alignat}
\end{subequations}
where we have used $\underline{I}$ and $\overline{I}$ to stand for $\min I$ and $\max I$ respectively.
Condition \eqref{so:condition1} states that an active segment should have all its modes included. Condition \eqref{so:condition2} states that all the modes that are located between modes of interest should not be included. Condition \eqref{so:condition3} states that any other segment that contains the active segment $I$ should not be active. Note that the set of indices of \eqref{so:condition2} and \eqref{so:condition3} might be empty, in which case they are satisfied vacuously.

Condition \eqref{so:condition1} is necessary because if a segment misses one of its modes, the MDT constraint is no longer applicable to that segment. Condition \eqref{so:condition2} is necessary because including any mode not included in $I$ separates the modes in $I$, making the MDT constraint inapplicable again. Condition \eqref{so:condition3} is necessary because if a segment that contains $I$ is active, then $I$ should not be active, as this will over-constrain the larger segment. These conditions are sufficient, as they ensure that always the largest segments of joined modes are active. 

As an example, we apply these conditions to two segments of \eqref{example:segments}. We have $z_{\{1,3\}} = 1$ if and only if:
\begin{alignat*}{2}
b_i &= 1, \qquad&&\forall i \in \{1, 3\},\\
b_j &= 0, \qquad&&\forall j \in \{2\}=\{1,2 ,3\}\setminus \{1, 3\}, \\
z_{\{1,3,5\}} &= 0.
\end{alignat*}
And $z_{\{1,3, 5\}} = 1$ if and only if:
\begin{subequations}
\label{example-dtc}
\begin{alignat}{2}
b_i &= 1, \qquad&&\forall i \in \{1, 3, 5\},\\
b_j &= 0, \qquad&&\forall j \in \{2, 4\}.
\end{alignat}
\end{subequations}
Note that in \eqref{example-dtc}, the third condition is vacuously satisfied. The logical conditions \eqref{so:conditions} can be imposed with the following set of linear inequalities:
\begin{subequations}
\label{so:condition:ineq}
\begin{alignat}{2}
z_I &\leq b_i, && \forall i \in I, \label{so:condition1:ineq}\\
z_I &\leq 1- b_j, &&\forall j \in [\underline{I}\, ..\, \overline{I}]\setminus I, \label{so:condition2:ineq}\\
z_I &\leq 1- z_{I'}, && \forall I' \in \I ,\;I' \supset I, \label{so:condition3:ineq} \\
z_I &\geq 1 - \Big[ \sum_{i \in I} (1 - b_i) + \sum_{j\in[\underline{I} .. \overline{I}] \setminus I}b_j  + \sum_{I' \in \I, I' \supset I}z_{I'}\Big]. \label{so:condition0:ineq}
\end{alignat}
\end{subequations}
Theorem~\ref{thm:equivalence} shows that \eqref{so:condition:ineq} implies that $z_I = 1$ is equivalent to \eqref{so:conditions}, and Theorem~\ref{thm:relaxed:z} shows that we can relax $z_I \in \{0, 1\}$ to $z_I \in [0, 1]$ as \eqref{so:condition:ineq} determines a unique value for $z_I$ in $\{0, 1\}$.
\begin{theorem}
Inequalities \eqref{so:condition:ineq} imply $z_I = 1$ is equivalent to \eqref{so:conditions}.
\label{thm:equivalence}
\end{theorem}
\begin{proof}
Suppose \eqref{so:condition:ineq} hold. If $z_I = 1$, then \eqref{so:condition1:ineq}, \eqref{so:condition2:ineq}, and \eqref{so:condition3:ineq} entail  \eqref{so:condition1}, \eqref{so:condition2}, and \eqref{so:condition3} respectively. 
Therefore, \eqref{so:condition1:ineq} to \eqref{so:condition3:ineq} capture the necessity of \eqref{so:condition1} to \eqref{so:condition3} for $z_I = 1$. On the other hand, \eqref{so:condition0:ineq} provides their sufficiency. This is because if \eqref{so:condition1} to \eqref{so:condition3} hold, then all the terms in the bracket on the right hand side of \eqref{so:condition0:ineq} become zero, leading to $z_I = 1$.
\end{proof}
\begin{theorem}
\label{thm:relaxed:z}
For all $I \in \I$, inequalities \eqref{so:condition:ineq} with $z_I \in [0, 1]$ imply that $z_I = 0$ if and only if $z_I \neq 1$.
\end{theorem}
\begin{proof}
To prove the theorem, we need to show that \eqref{so:condition:ineq} implies the following equivalence: $z_I = 0$ if and only if one of the conditions in \eqref{so:conditions} fails to hold. Since, by Theorem~\ref{thm:equivalence}, \eqref{so:conditions} is also equivalent to $z_I = 1$, the theorem follows. 

First, we show that if $I$ is a segment such that the theorem holds for all its supersets $I' \in \I$, then the theorem holds for $I$. Suppose that $I$ is a segment such that for all its supersets $I' \in \I$ the theorem holds for $I'$. Then, if \eqref{so:condition1} does not hold, \eqref{so:condition1:ineq} leads to $z_I = 0$. Similarly, violation of \eqref{so:condition2} with \eqref{so:condition2:ineq} entails $z_I = 0$. If \eqref{so:condition3} is not satisfied, i.e., $z_{I'} \neq 0$, by assumption this is equivalent to $z_{I'} = 1$. By \eqref{so:condition3:ineq}, this implies $z_I = 0$. Therefore, if any of the conditions \eqref{so:conditions} is not satisfied, $z_I = 0$. On the other hand, $z_I = 0$ with \eqref{so:condition0:ineq} implies that at least one of the sums on the right hand side of  \eqref{so:condition0:ineq} is positive, and thus at least one of the conditions of \eqref{so:conditions} is not satisfied. Therefore, if $I$ is a segment such that the theorem holds for all its supersets $I' \in \I$, then the theorem holds for $I$. Next, we need to show that the theorem holds for all $I \in \I$. We will use $\subseteq$ as a partial ordering on $\I$. 

Let $\hat{\I} \subseteq \I$ be the set of all segments such that the theorem holds for them. This set is nonempty because there exists a maximal element $I \in \I$ for which there exists no segment $I' \in \I$ such that $I' \supset I$. Therefore, the condition of the argument above is vacuously satisfied, and the theorem holds for it. We now show that $\I \subseteq \hat{\I}$, resulting in $\hat{\I} = \I$, i.e., the theorem holds for all $I \in \I$.

Suppose $\I \not\subseteq \hat{\I}$. It follows that the set $\I \setminus \hat{\I} \neq \emptyset$ and the theorem does not hold for any of its members. However, $\I \setminus \hat{\I}$ is a finite set that has a maximal element $I$. Since $I$ is the maximal element, there are no $I' \in \I \setminus \hat{\I}$ such that $I' \supset I$. It follows that for any $I' \in \I$ such that $I' \supset I$, it cannot belong to $ \I \setminus \hat{\I}$. It follows that $I' \in \hat{\I}$, and by definition of $\hat{\I}$, the theorem holds for $I'$. Consequently, $I$ is a segment such that for any of its supersets $I' \in \I$ the theorem holds for $I'$. It follows that theorem holds for $I$, contradicting our assumption that the theorem does not hold for any members of $\I \setminus \hat{\I}$. Therefore, $\I \subseteq \hat{\I}$.
\end{proof}
The MDT constraint \eqref{sto:dtc} is now modified using the newly introduced variable $z_I \in [0, 1]$ to the following constraint:
\begin{subequations}
\label{sto:dtc:binary}
\begin{alignat}{2}
\sum_{k \in I} w_k \geq z_I \D, \qquad&& \forall I \in \I(\v).  \\
0 \leq z_I \leq 1 , \qquad&& \forall I \in \I(\v).
\end{alignat}
\end{subequations}
The inequalities \eqref{so:condition:ineq} along with \eqref{sto:dtc:binary} are collected as $G(w, b, z_I) \geq 0$, which will bring MDT constraints back to \eqref{sto:binary}, leading to the following continuous optimization problem:
\begin{subequations}
\begin{alignat}{2}
\phi(b) = \min_{\substack{x, u, \\ z_I, w}} &\sum_{k=1}^{\M} L(x_k, u_k, \v_k, w_k) + E(x_{\M+1}) \\
\text{s.t.}\hspace{0.9cm} x_1 &= \bar{x}_1,\\
x_{k+1} &= F(x_k, u_k, \v_k, w_k), &&\hspace{-1cm} \forall k \in [\M],   \\
\sum_{k=1}^{\M} w_k &= t_f, \\
0 \leq  w_k &\leq b_k T, &&\hspace{-1cm} \label{sto:w:complete}  \forall k \in [\M], \\
G(w, b, z_I) &\geq 0, &&\hspace{-1cm}\label{sto:dtc:complete} \forall I \in \I(\v),
\end{alignat}
\label{sto:binary:complete}%
\end{subequations}
where for all $k \in [M]$, $x_k \in \R^{n_x}$, $u_k \in \R^{N_kn_u}$, and $w_k \in \R_{\geq 0}$. The sequence $\v$ is the fixed master sequence, and $b$ is a binary vector telling which modes of $\v$ are included. Note that for all $I \in \I(\v)$, $z_I \in [0, 1]$, and in principle $z_I$ could be eliminated from the above optimization problem as it is uniquely determined by $b$.

\section{ITERATIVE SWITCHING TIME Optimization}
\label{sec:isto}
The optimization problem $\min_{b \in \{0,1\}^{\M}} \phi(b)$  with $\phi(b)$ defined in \eqref{sto:binary:complete} is intended to find an optimizer $b^*$, indicating an optimal $v^*$ as a subsequence of the master sequence $\v$. It is possible to devise an efficient iterative algorithm, which we will refer to as iterative switching time optimization (ISTO), to provide a good heuristic solution to the STO problem \eqref{sto} and the SO problem \eqref{so}, without introducing binary variables. The ISTO algorithm presented here generalizes a basic version that was introduced in a previous publication\cite{abbasi-esfeden_iterative_2023}, which considered only simple minimum dwell time constraints. Also, here we use a different selection procedure for $v$.

Provided with a master sequence $\v$, evaluation of $\Phi(\v)$ requires solving the continuous optimization problem \eqref{sto}. In the absence of any dwell time constraint, it is easy to suggest a $v^*$ as follows. Let $\Phi_{\mathrm{exMDT}}(\v)$ be the STO problem excluding MDT constraints. Let $\bar{w}$ be the optimal dwelling times for $\Phi_{\mathrm{exMDT}}(\v)$, and $\K= \{k | \bar{w}_k > 0 \}$. Define $\one_\K$ to be the submatrix of the identity matrix $\one$ by keeping the rows $i \in \K$. If we represent a sequence as a column matrix, by identifying the $k$th row of the matrix with the $k$th element of the sequence, then $v^* := \one_\K \v$, with the corresponding optimal dwell times $w^* := \one_\K \bar{w}$, will be a subsequence of $\v$ with $\Phi_{\mathrm{exMDT}}(v^*) = \Phi_{\mathrm{exMDT}}(\v)$ with no zero dwell time for any of its modes. However, as mentioned in the introduction, $\bar{w}$ may no longer be feasible once MDT constraints are added. ISTO is an algorithm with the aim of finding a $v^*$ with $\Phi(v^*)$ close to $\Phi_{\mathrm{exMDT}}(\v)$.
\subsection{ISTO Algorithm}
Since MDT constraints will prevent modes from collapsing to zero duration, we need to soften the constraints \eqref{sto:dtc} by introducing a slack variable $e \in \R_{\geq 0}^{M}$. This modifies \eqref{sto:dtc} to:
\begin{alignat}{2}
\sum_{k \in I} w_k + e_k \geq \D, \qquad&& \forall I \in \I(v), \label{sto:dtc:soft}
\end{alignat}
and adds a quadratic cost $\gamma_0 e^\top e$ to the objective function. Alternatively, a $l_1$-norm cost could also have been used. However, a quadratic cost, by penalizing larger slacks more severly, ensures that violation of MDT constraints are more uniformly distributed among modes, leading to a better performance of the algorithm we experienced for the numerical results. In \eqref{sto:binary:complete}, whether a constraint is enforced or removed was decided through the binary variables $b$ and $z_I$.  The ISTO algorithm implements the same decision using the complementarity constraint $0 \leq e \perp w \geq 0$. This is motivated by \eqref{sto:dtc:soft}, where the values of the two vectors of $w$ and $e$ affect MDT constraints in the following way: setting $e_k = 0$ for all $k \in I$ enforces a constraint, and setting $w_k = 0$ for any $k \in I$ removes that constraint if we remove the corresponding $k$th mode from $v$. As a result, for a sequence $v$ with length $M$, ISTO replaces \eqref{sto} with
\begin{subequations}
\begin{alignat}{2}
\iPhi(v) &=\min_{\substack{x, u \\ w, e}} \sum_{k=1}^{M} L(x_k, u_k, v_k, w_k) + E(x_{M+1}) + \gamma_0 e^T e \\
\text{s.t.}\hspace{0.4cm} &x_1 = \bar{x}_1, \label{isto:dyn}\\
&x_{k+1} = F(x_k, u_k, v_k, w_k), &&\hspace{-2cm} \forall k \in [M],   \\
&\sum_{k=1}^M w_k = t_f, \\
&\sum_{k \in I} w_k + e_k \geq \D, &&\hspace{-2cm}   \forall I \in \I(v), \label{isto:dtc}\\
&0 \leq e \perp w \geq 0. \label{isto:cc}
\end{alignat}
\label{sto:isto}%
\end{subequations}
 The complementarity constraint itself can be incorporated using the penalty function $\sum_{k=1}^{M} \gamma e_k w_k$, resulting in the following continuous optimization problem:
\begin{align*}
\iPhi^{\gamma}(v) = \min_{\substack{x, u \\ w, e}} & \sum_{k=1}^{M} \Big[ L(x_k, u_k, v_k, w_k) + \gamma e_k w_k + \gamma_0 e_k^2\Big] +  E(x_{M+1}) \\[0.2cm]
\text{s.t.}\hspace{0.9cm} &\eqref{isto:dyn} - \eqref{isto:dtc}, \; e \geq 0 , \;w  \geq 0.
\end{align*}
The key idea behind ISTO is as follows. Start by solving $\iPhi^\gamma(\v)$ for a $\gamma \approx 0$. Note that at this point, $\iPhi^\gamma(\v) \approx \Phi_{\mathrm{exMDT}}(\v)$, as the MDT constraints are softened by a small $\gamma_0$ coefficient. Next, increase the penalty parameter $\gamma$, and further decrease $\gamma_0$ to allow modes to collapse more easily. If this causes a mode $k$ to collapse, define $\one_\K$ based on $\K=\{k | \bar{w}_k > 0 \}$ and by representing the sequence $\v$ as a column matrix, update the sequence by $v' = \one_\K \v$. Repeat the process by solving $\iPhi^\gamma(v')$, preferably warm-started from the dwell times of the previous solution. Note that since $v'$ has a different dimension than $\v$, $\iPhi^\gamma(v')$ will have a different structure and different variables than $\iPhi^\gamma(\v)$. The hope is that by slowly increasing $\gamma$, while keeping $\gamma_0$ small, the MDT constraints will be settled either in favor of removing them (by removing the mode) or satisfying them (by setting the slacks to zero) such that we remain close to $\Phi_{\mathrm{exMDT}}(\v)$. Details of the ISTO iterations are provided in Algorithm~\ref{isto}. Note that in Algorithm~\ref{isto}, $\iPhi^\gamma(v)$ is solved multiple times (lines \ref{isto:solve1}, \ref{isto:solve2}, and \ref{isto:solve3}), each time updating the solution vectors $w^*$ and $e^*$. To use ISTO algorithm in Section.~\ref{sec:results}, at line \ref{isto:parameters} of ISTO, we set the initial value of $\gamma$ and $\epsilon$ to $10^{-4}$. $\gamma_0$ is initially set to $1$, and further reduced to $\gamma_0' = 10^{-2}$ in line \ref{isto:gamma}. The penalty growth factor $\theta$ is set to $10$.

Definition~\ref{defn:licq} describes the constraint qualification that is assumed in Lemma~\ref{mpcc-licq} and Theorem~\ref{thm:convergence} to prove the termination of the Algorithm~\ref{isto}.

\begin{algorithm}
\begin{algorithmic}[1]
\State Receive $v^{(1)} = \v = (\v_1, \v_2, \ldots, \v_{\M})$. 
\State $i \rightarrow 1$. 
\State Set the initial values for $\gamma$, $\gamma_0$, $\theta$, and $\epsilon$. \label{isto:parameters}
\State Compute $\iPhi^{\gamma}(v^{(i)})$, receive $w^*$ and $e^*$. \label{isto:solve1}
\While {TRUE}
\If {$\{k | { w^*_k \leq \epsilon \vee e^*_k \geq \epsilon }\} = \emptyset$} \label{isto:end}
\State End.
\Else
\State $\gamma \rightarrow \theta \gamma$, $\gamma_0 \to \gamma_0'$. \label{isto:gamma}
\State Compute $\iPhi^{\gamma}(v^{(i)})$, receive $w^*$ and $e^*$.  \label{isto:solve2}
\If {$\{k |{ w^*_k \leq \epsilon}\} \neq \emptyset$}
\State Let $\K=\{k |{ w^*_k > \epsilon }\}$
\State Define $v^{(i+1)} = \one_\K v^{(i)}$ \label{isto:v}
\State $i \rightarrow i + 1$. \label{isto:i}
\State Set the initial values for $\gamma$, $\gamma_0$.
\State Compute $\iPhi^{\gamma}(v^{(i)})$, receive $w^*$ and $e^*$. \label{isto:solve3}
\EndIf
\EndIf
\EndWhile
\end{algorithmic} 
\caption{Iterative Switching Time Optimization}
\label{isto}
\end{algorithm}

\begin{definition}
\label{defn:licq}
Let $\hat{\xi} = (\hat{x}, \hat{u}, \hat{w}, \hat{e})$ satisfy all constraints of \eqref{sto:isto} with the possible exception of \eqref{isto:cc}. Infeasible-point MPCC-LICQ \cite{Hu2004} holds at $\hat{\xi}$ if the following gradients are linearly independent at $\hat{\xi}$ :
\begin{subequations}
\begin{alignat}{2}
&\nabla_\xi \big(x_1 - \bar{x}_1\big), \label{licq:x1}\\
&\nabla_\xi \big(x_{k+1} - F(x_k, u_k, v_k, w_k)\big),&&\hspace{0.5cm} \forall k \in [M], \label{licq:xk}\\
&\nabla_\xi \big(\sum_{k \in I'} w_k + e_k - \D \big),&&\hspace{0.5cm} \forall I' \in \hat{\I}(v),  \label{licq:dtc}\\
&\nabla_\xi \big(\sum_{k=1}^M w_k - t_f\big), \label{licq:wsum}\\
&\nabla_\xi w_i, \;\;  k: w_k = 0 , &&\hspace{0cm}\nabla_\xi e_k, \;\; k : e_k = 0, \label{licq:we1}\\ 
&\nabla_\xi w_i,  \;\; k : e_k, w_k > 0 , &&\hspace{0cm}\nabla_\xi e_k, \;\; k : e_k, w_k > 0, \label{licq:we2}
\end{alignat}
\end{subequations}
where $\nabla_\xi (\cdot)$ is the gradient (Jacobian in cases of \eqref{licq:x1} and \eqref{licq:xk}) with respect to $\xi = (x, u, w, e)$, and $\hat{\I}(v) \subset \I(v)$ is the set of segments whose MDT inequality constraint has become active. { The infeasible-point MPCC-LICQ reduces to the usual MPCC-LICQ if $\hat{\xi}$ is feasible for \eqref{sto:isto}.}
\end{definition}
\begin{lemma}
For a $v \in [n]^M$, let the solution of $\iPhi^\gamma(v)$ converge to { a limit point} $\hat{\xi}$ as $\gamma \rightarrow \infty$. If the infeasible-point MPCC-LICQ holds at $\hat{\xi}$, then $\hat{\xi}$ is a feasible point for \eqref{sto:isto}. Moreover, $\hat{\xi}$ is a C-stationary point of \eqref{sto:isto}.
\label{mpcc-licq}
\end{lemma}
\begin{proof}
See Hu and Ralph \cite{Hu2004}, Theorem 2.1 and Lemma 3.2.
\end{proof}

Using Lemma~\ref{mpcc-licq}, Theorem~\ref{thm:convergence} shows that under the infeasible-point MPCC-LICQ, the ISTO algorithm terminates in finite steps, and it ends with a solution that satisfies the MDT constraints {up to any given accuracy.}
\begin{theorem}
\label{thm:convergence}
For a $v \in [n]^M$, let the solution of $\iPhi^\gamma(v)$ converge to { a limit point}. { If the infeasible-point} MPCC-LICQ holds at this limit point, Algorithm~\ref{isto} terminates after $m < \infty$ number of iterations with a sequence $v^* := v^{(m)}$ { with $w_k^* > \epsilon$, and $e_k^* < \epsilon$ for any mode $k$ and any given $\epsilon > 0$}.
\end{theorem}
\begin{proof}
{ Let $\epsilon > 0$ be a small arbitrary number.} Suppose the algorithm does not terminate. This leads to either indefinite execution of line~\ref{isto:v} or line~\ref{isto:gamma} or both. Line~\ref{isto:v} reduces the dimension of $v$ as $\one_\K$ always removes at least one row. Since $\v$ { has} finite dimensions, Line~\ref{isto:v} cannot be executed indefinitely. Therefore, $i$ will reach a maximum $m < \infty$, and the sequence will remain fixed at $v^{(m)}$ of length $\tilde{M}$. Subsequently, line~\ref{isto:gamma} will be executed indefinitely for this sequence. This means $\gamma \rightarrow \infty$. On the one hand, since the algorithm does not terminate, there exists a $k \in [\tilde{M}]$, such that { $w^*_k \leq \epsilon$ or $e^*_k \geq \epsilon$}. On the other hand, since line~\ref{isto:v} is not executed, { $w_k > \epsilon$} for all $k \in [\tilde{M}]$. It follows that there exists a $k \in [\tilde{M}]$ such that { $e_k^* \geq \epsilon$}. In conclusion, { if the algorithm does not terminate,} there is a $k \in [\tilde{M}]$ such that { $e^*_k w^*_k > \epsilon^2 > 0$}, regardless of how large $\gamma$ is. {However, assuming that the limit point of $\iPhi^\gamma(v^{(m)})$ satisfies the infeasible-point MPCC-LICQ, it is a feasible point of $\iPhi(v^{(m)})$ by Lemma~\ref{mpcc-licq}. Therefore, for any $\epsilon>0$, there exists a $\gamma$ such that $e_k^*w_k^* < \epsilon^2$.} { It follows}, by contradiction, that the algorithm terminates. Furthermore, the algorithm terminates only if ${\{  w^*_k \leq \epsilon \vee e^*_k \geq \epsilon} \}= \emptyset$, i.e., there exists a maximum iteration $m$ and a sequence $v^{(m)} =: v^*$, such that for any of its modes $k$, { $w^*_k > \epsilon$ and $e^*_k < \epsilon$}. 
\end{proof}

\begin{remark}
Theorem \eqref{thm:convergence} assumes that the infeasible-point MPCC-LICQ holds at a point $\hat{\xi}$ that ISTO algorithm converges to, should it fail to terminate. In this research, we will not attempt to prove this assumption. However, note that gradients \eqref{licq:x1} \eqref{licq:xk} are linearly independent by structure, and \eqref{licq:dtc} to \eqref{licq:we2} only consist of ones and zeros for rows corresponding only to $w$ and $e$. Therefore, the investigation of MPCC-LICQ is limited only to $w$ and $e$ vectors appearing in \eqref{licq:dtc} to \eqref{licq:we2}.
\end{remark}
\begin{remark}
In this paper, state constraints have been omitted in the formulations of STO. Although $\Phi(v)$ can include state constraints, the termination of the algorithm can no longer be guaranteed.
\end{remark}

\section{NUMERICAL RESULTS}
\label{sec:results}

In this section, we apply the developed MINLP formulation \eqref{sto:binary:complete}, with $b$ as the binary variables to be optimally selected using branch-and-bound, and the ISTO algorithm to four switched systems with minimum dwell time constraints. Three of these are benchmark problems \cite{lee_benchmark_2012}: the Double Tank System (DTS), the Lotka-Volterra fishing binary problem (LVF), and the Van der Pol oscillator (VDP). The fourth system is a simple linear system described by:
\begin{subequations}
\begin{align}
\min_{\substack{x, \mathrm{v}}} & \int_0^{10} \Big(x_1(t) - \big(0.5\sin(t) + 1\big) \Big)^2 dt,  \\ 
\text{s.t.}\hspace{0.9cm}  \dot{x}(t) &=
\begin{bmatrix}
0 & 1\\
0 & 0
\end{bmatrix} x(t) + 
\begin{bmatrix}
0 \\
1
\end{bmatrix} \mathrm{v}(t), \\
x(0) &= \bar{x}_0,  
\end{align}
\end{subequations}
with $\bar{x}_0 = [0, \; 0]^\top$, $x(t) = [x_1(t), \;x_2(t)]^\top \in \R^{2}$, and $\mathrm{v}(t) \in \{0, 1, -1\}$ for almost all $t \in [0, 10]$. We will refer to it as the Particle Trajectory (TRJ) problem.

Three approaches are compared: The Combinatorial Integral Approximation (CIA) \cite{sager_combinatorial_2011}, the MINLP formulation of \eqref{sto:binary:complete} and ISTO. To obtain a lower bound, a relaxation of the optimal control problem of system \eqref{originalsys} is performed by allowing the discrete control input to take on continuous values from the convex hull of the set of discrete values, i.e., $\mathrm{v}(t) \in \{-1, 0, 1\}$ is relaxed to $\mathrm{v}(t) \in [-1, 1]$. This transforms the original optimal control problem into a continuous optimization problem. The MDT constraints do not apply to the relaxed problem.

CIA starts by solving the relaxed problem, written in the outer convexification form \cite{sager_numerical_2005}. It then approximates the continuous input values by discrete values, considering the MDT constraints. This step is performed using {\ttfamily pycombina}, which implements a tailored branch-and-bound algorithm\cite{burger_pycombina_2020}. Afterward, by fixing the discrete values, the optimization problem is resolved. The mixed integer problem of $\min_{b \in \{0,1\}^{\M}} \phi(b)$ with $\phi(b)$ defined in \eqref{sto:binary:complete} is solved using the solver Bonmin \cite{Bonmin}, which employs a branch-and-bound algorithm. The continuous optimization problems within Algorithm~\ref{isto} were solved using the interior point solver IPOPT \cite{wachter_implementation_2006}.

None of the problems include state constraints, and a minimum dwell time constraint is imposed on all modes. The master sequence for all problems consists of 10 elements, and it is generated by iterating over the possible discrete values of the input. Note that this means that the only integer variable is the vector $b \in \{0,1\}^{10}$, regardless of the number of discretization nodes.  The master sequence and the MDT constraint lower bound values are collected in the table below:
\begin{table*}[!ht]%
\centering %
\begin{tabular*}{0.4\textwidth}{@{\extracolsep\fill}ccc@{\extracolsep\fill}}
\toprule
\textbf{Name} & \textbf{Master Sequence}  & \textbf{$\D$ } \\
\midrule
DTS & $[1, 2, 1, 2, 1, 2, 1, 2, 1, 2]$ & 0.5\\
LVF & $[0, 1, 0, 1, 0, 1, 0, 1, 0, 1]$ & 0.5\\
VDP & $[-2, -1, 0.75, -2, -1, 0.75, -2, -1, 0.75, -2]$ & 1\\
TRJ & $[-1, 0, 1, -1, 0, 1,-1, 0, 1, -1]$ & 0.5\\
\bottomrule
\end{tabular*}
\end{table*}

Time discretization is performed using multiple shooting and an explicit Runge-Kutta of order 4 as the integrator. Nodes are distributed across modes according to the initial dwell time values, with the goal of keeping the initial lengths of discretization intervals within each mode similar. The initial dwell time can be set uniformly across all modes, or in the case of ISTO, be determined by the solution from the previous iteration. The continuous control inputs are assumed to remain constant over each time discretization node. 

Since the integration step in STO depends on the dwell time vector $w$, while CIA is discretized on a fixed time grid, to ensure a fair comparison of the objective function, all systems are forward-simulated using the calculated optimal control input on the same time grid using the same integrator. The objective function is then computed based on the simulation. The problems are solved for a range of number of time discretization nodes, and the results are plotted in Fig.~\ref{fig:dts-trj}. The top plot shows the optimal value of the objective function and the bottom plot displays the single-core~\footnote{With CPU specification: Intel® Core™ i5-1245U × 12, and 16 GB RAM.} computation time ($\mathrm{t_{proc}}$) as a function of time discretization nodes $N$. Each problem instance is solved 10 times, and the average is reported. The results demonstrate the efficiency of Bonmin in solving \eqref{sto:binary:complete} even for a large number of nodes. This is because the number of binary variables is fixed at 10. The ISTO algorithm provides similar solutions to Bonmin but with significantly lower computation times. Due to the nonconvexity of the STO formulation, in some cases, small discrepancies are observed between the objective values of Bonmin and ISTO. Lastly, although CIA, as an integer approximation method, is faster than ISTO, the optimality gap due to the minimum dwell time constraints results in solutions with higher objective values.
\begin{figure}[h!]
\centering
\includegraphics[width=0.23\textwidth, trim = 0 0.8cm 2cm 1.5cm, clip]{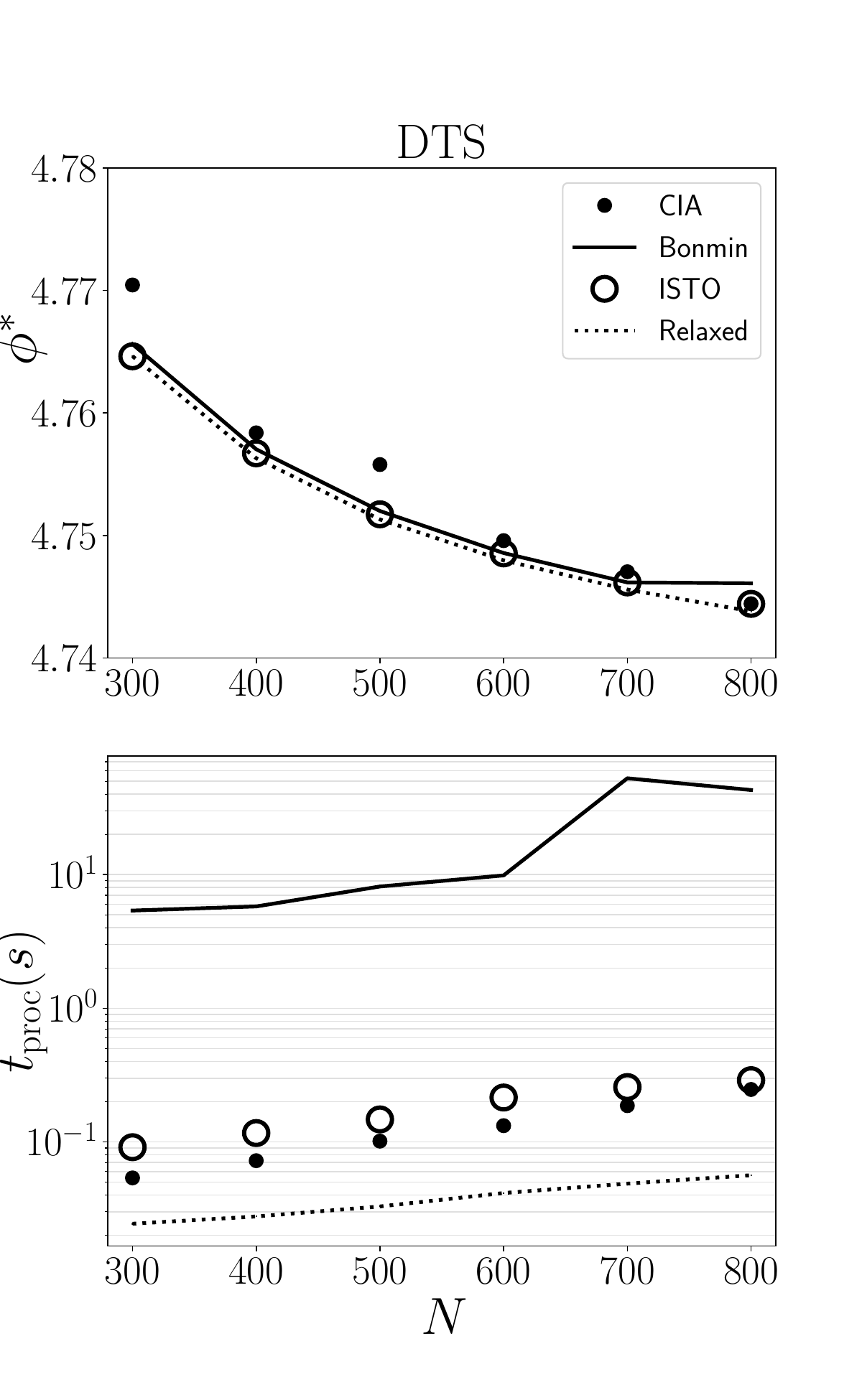}
\includegraphics[width=0.23\textwidth, trim = 0 0.8cm 2cm 1.5cm, clip]{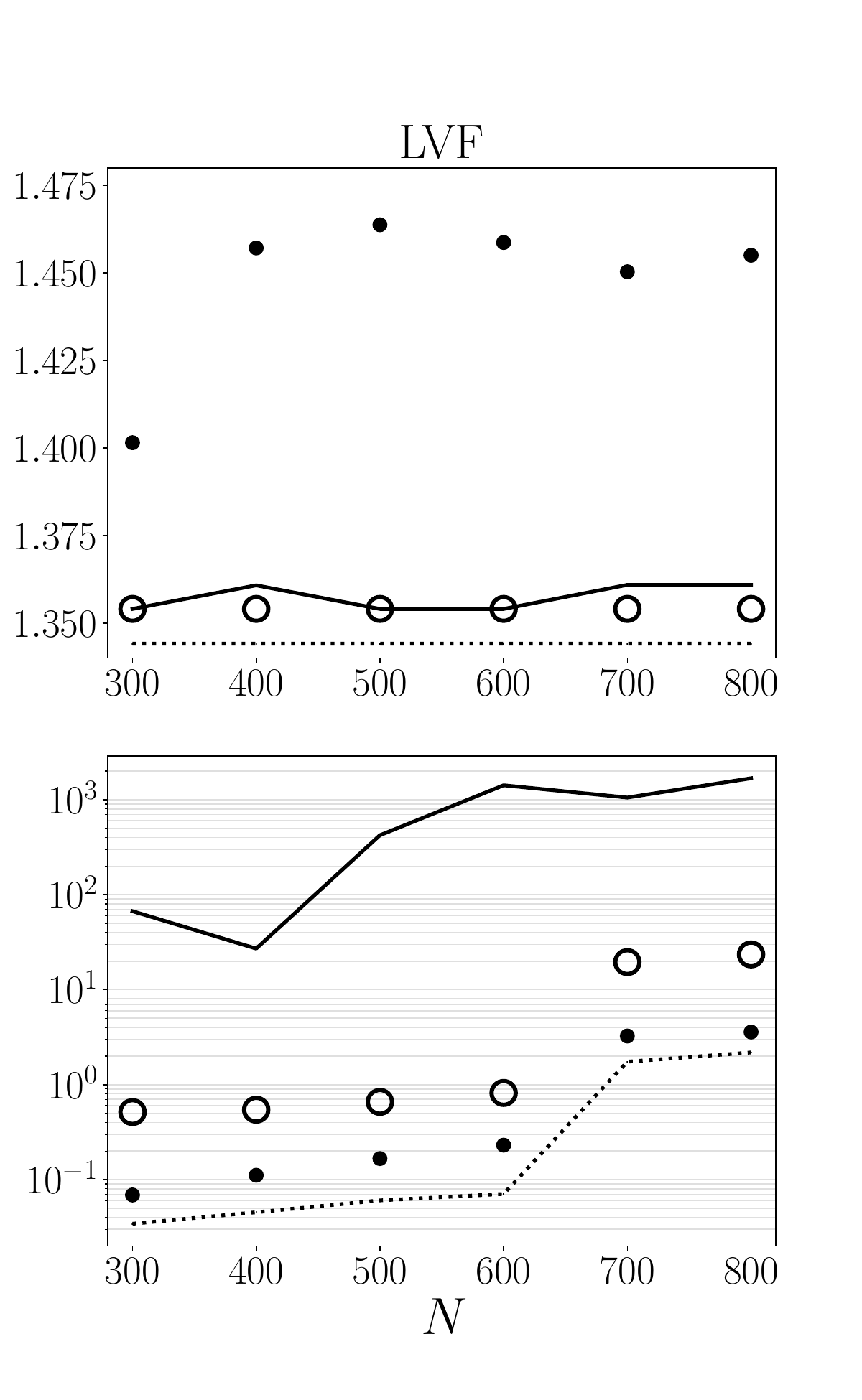}
\includegraphics[width=0.23\textwidth, trim = 0 0.8cm 2cm 1.5cm, clip]{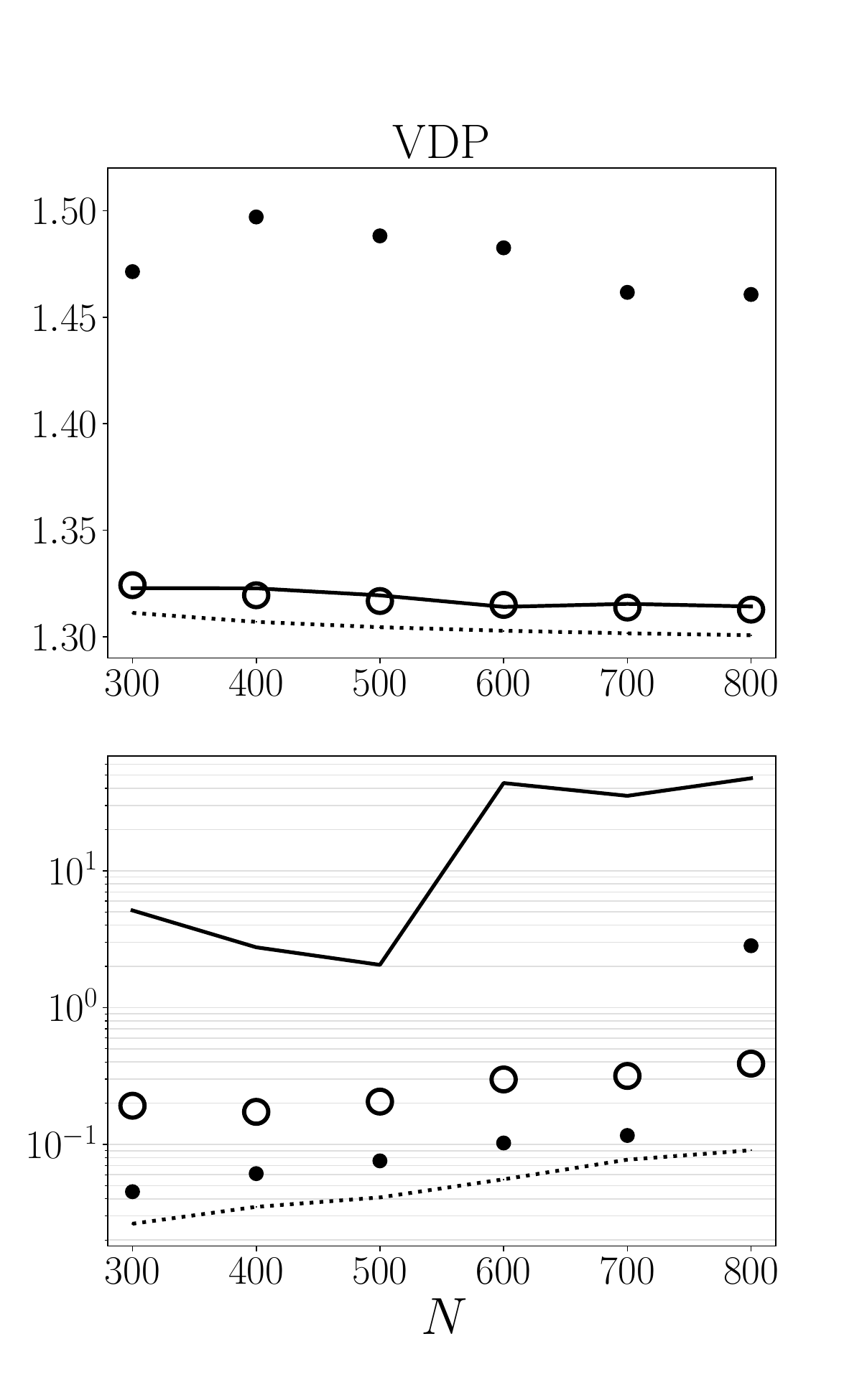}
\includegraphics[width=0.23\textwidth, trim = 0 0.8cm 2cm 1.5cm, clip]{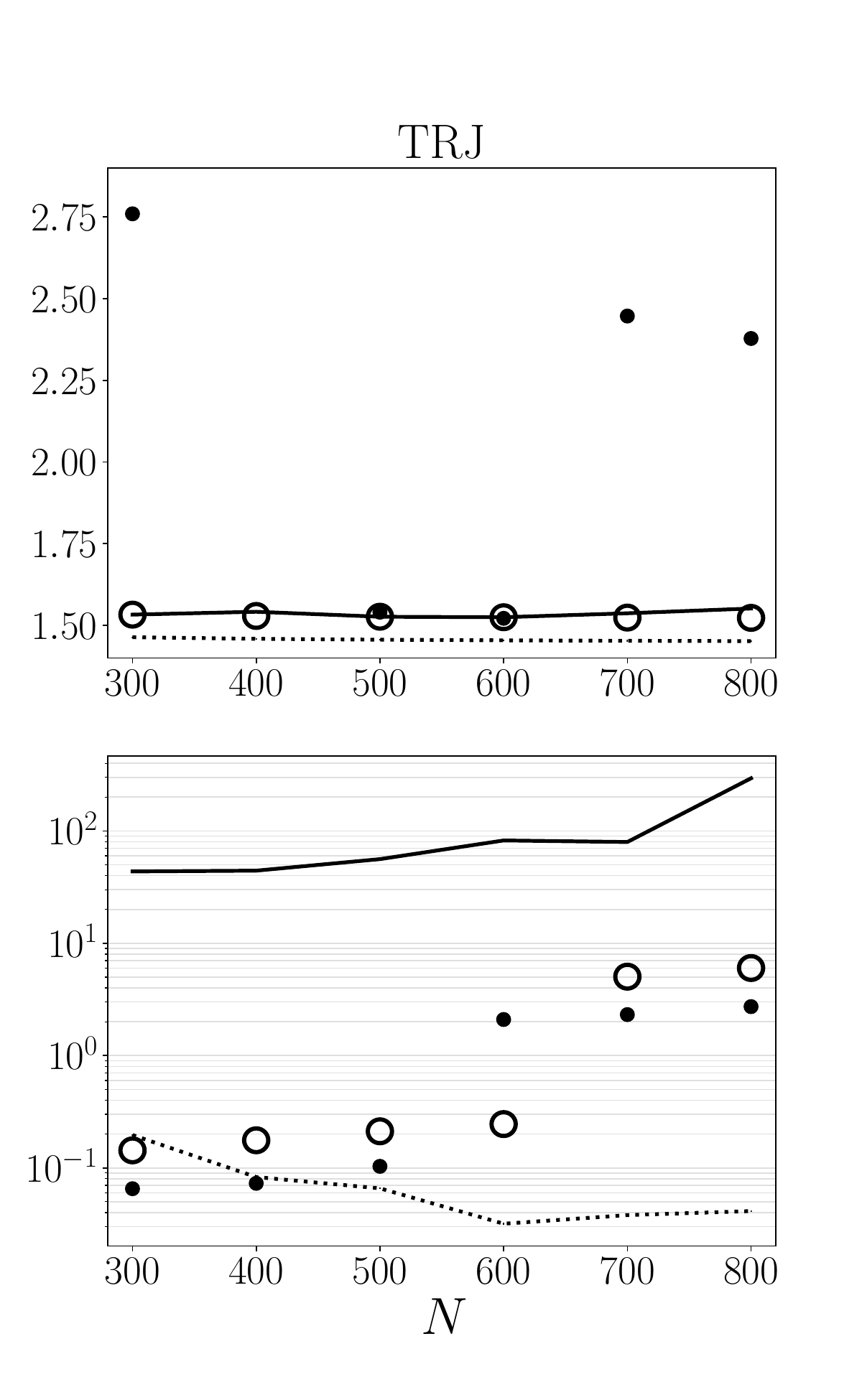}
\caption{From left to right: the Double Tank problem, the Lotka Volterra Fishing binary problem, the Van Der Pol oscillator, and the Particle Trajectory problem.}
\label{fig:dts-trj}
\end{figure}

\section{CONCLUSIONS}
We presented a MINLP formulation for systems with discrete control inputs under minimum dwell time constraints. The problem was decomposed into a Switching Time Optimization (STO) part and a Sequence Optimization (SO). The feasible set of SO was limited to subsequences of a master sequence, and minimum dwell time (MDT) constraints were incorporated into STO. This resulted in a MINLP formulation that could be solved by a branch-and-bound solver. Since in this formulation, the number of binaries depends only on the master sequence, it was proved to be an efficient formulation. Moreover, the iterative switching time optimization (ISTO) algorithm was introduced to efficiently solve STO and SO while respecting the MDT constraints. Both approaches were applied to four switched systems and the results show the efficiency of the MINLP formulation and ISTO.


\bibliography{wileyNJD-AMA}




\end{document}